\documentclass[12pt, a4paper]{amsart}
\usepackage{mathrsfs, amssymb, eucal, amsfonts, amsmath, amsthm, amscd, enumerate}

\providecommand{\gram}[2]{\left[ #1 , #2 \right]} \providecommand{\scal}[2]{\left\langle #1 , #2 \right\rangle}
\newcommand{\Bc}{{\mathcal B}}\newcommand{\Cc}{{\mathcal C}}
\newcommand{\Dc}{{\mathcal D}}
\newcommand{\Gc}{{\mathcal G}}

\newcommand{\Lc}{{\mathcal L}}\renewcommand{\Mc}{{\mathcal M}}
\newcommand{\Rc}{{\mathcal R}}\newcommand{\Tc}{{\mathcal T}}

\newcommand{\Zc}{{\mathcal Z}}\newcommand{\As}{{\mathscr A}}
\newcommand{\Cs}{{\mathscr C}}\newcommand{\Ds}{{\mathscr D}}

\newcommand{\Rs}{{\mathscr R}}
\newcommand{\Ls}{{\mathscr L}}\newcommand{\Ss}{{\mathscr S}}

\newcommand{\C}{{\mathbb C}}\newcommand{\D}{{\mathbb D}}
\newcommand{\N}{{\mathbb N}}
\newcommand{\R}{{\mathbb R}}

\newcommand{\bA}{\boldsymbol{A}}\newcommand{\bB}{\boldsymbol{B}}
\newcommand{\ba}{{\mathbf a}}

\numberwithin{equation}{section}

\newtheorem{thm}{\sc Theorem}[section]
\newtheorem{lem}[thm]{\sc Lemma}
\newtheorem{prop}[thm]{\sc Proposition}
\newtheorem{cor}[thm]{\sc Corollary}
\theoremstyle{remark}
\newtheorem{defn}{\sc Definition}[section]
\newtheorem{ex}{\bf Example}[section]
\newtheorem{rem}{Remark}[section]
\begin{document}
\title{On selfadjoint spectral theory of random operators}
\author{P\u astorel Ga\c spar}
\address{Department of Mathematics and Computer Science,
Faculty of Exact Sciences,
``Aurel Vlaicu'' University, Arad,
Str. E. Dr\u agoi nr. 2,
310330  Arad,
Romania }
\email{pastorel.gaspar@uav.ro}
\date{\today}
\begin{abstract}
In this paper spectral theorems for not necessarily continuous normal and self-adjoint random operators on a complex separable Hilbert space are proved.
\end{abstract}
\subjclass[2010]{60G60 ; 46F12; 42B10}
\keywords{stochastic mappings, random operator, unitary random operator, random Hermitian operator, random projection operator valued measures, spectral representation}
\maketitle

\section{Introduction}

In quite recent times the study of stochastic processes or random fields was enlarged to the framework of multivariate stochastic mappings (see \cite{ThangRndMapInfDimSp98}, \cite{GaPopa}) in order to treat in a unitary way also other probabilistic concepts such as {\it stochastic measures} and {\it stochastic integrals}, {\it random distributions} or {\it random distribution fields}, as well as {\it random operators} (see \cite{ThangSerRepr98}, \cite{ThangRndOpBan87}, \cite{ThangRndMapInfDimSp98}, \cite{ThangAdjCompRndOp95}, \cite{ThangProcExtRndOp09}, \cite{ThangRndBndOpKyu98}, \cite{ThangRndOptoRndBndOp08}, \cite{Sko}, \cite{GaPopaII}, \cite{GaPopaIII}, \cite{PGaLavSid}, \cite{Hack5}, \cite{Hack4}), but also in an attempt to develop in this setting a corresponding random spectral theory (see \cite{Sko}, \cite{Thang4}, \cite{ThangAdjCompRndOp95}, \cite{Hack5}, \cite{Hack4}).

Specifically, in \cite{Thang4}, in terms of measurable families of deterministic continuous linear operators, the continuous normal and Hermitian random operators, random spectral measures were defined and the random version of the spectral (integral) representation theorems were given.

On the other hand, intending to cover the important class of random Schr\"odinger operators, W. Hackenbroch extended the concept of continuous random operator, to densely defined random operator, treating in \cite{Hack4} and \cite{Hack5} some aspects of random spectral theory for symmetric densely defined random operators.

We aim here to continue the work from \cite{Thang4} and \cite{Hack5} developing a stochastic selfadjoint spectral theory for not necessarily continuous random operators, by using the theory of measurable fields of unbounded operators on Hilbert spaces.

The plan of the paper runs as follows. The remainder of this Section contains some general concepts and basic results especially regarding continuous random operators.

In Section 2 we select from \cite{Schmud}, \cite{Hack4} and \cite{Thang4} some specific results (regarding decomposable random operators and the random adjoint) reformulated and completed in a way to suit our present development.

In Section 3 we define random normality, together with the integral of a bounded and then of an unbounded measurable function with respect to a random projection operator valued measure, obtaining in this way (not only continuous) random normal operators.

Section 4 deals with transforming random operators into continuous random operators and viceversa aiming to extend in Section 5 the results from \cite{Thang4} regarding the spectral theorem for continuous random normal operators to the unbounded framework.

For two linear topological spaces $X$ and $Y$, $\Bc(X, Y)$ means the space of continuous linear operators from $X$ to $Y$, while $\Lc (X, Y)$ denotes the set of linear operators $T: D(T) \to Y$, where $D(T)$ is some linear subspace of $X$. When $T$ is densely defined, i.e. $D(T)$ is dense in $X$, we shall write $T\in \Lc_d(X, Y)$ and when $T$ is closed, i.e. the graph ${\mathcal G}_T :\ = \{ (x, Tx),\ x \in D(T) \}$ of $T$ is closed in the topological linear product $X \times Y$, we use the notation $T\in \Cc\Lc(X, Y)$. When $T$ is densely defined and closed we write $T\in \Cc\Lc_d (X, Y)$. Note that when, for an operator $S \in \Lc (X, Y)$, $\Gc_S$ is not necessarily closed in $X \times Y$, but its closure $\overline{\Gc_S}$ is the graph of an operator, denoted by $\overline{S}$, then $S$ is said to be \emph{closable} and $\overline{S}$ is called its {\it closure}. Moreover, when $T \in \Cc\Lc (X, Y)$, a subspace $\Dc \subset \Dc(T)$ is a \emph{core} for $T$, when $T = \overline{T|_\Dc}$

Now $(\Omega, \As, \wp)$ denotes a fixed complete separable probability space, while $(H, \langle \cdot, \cdot \rangle_H)$ a fixed separable, complex Hilbert space. The most frequently used spaces of $H$-valued random variables (measurable functions) in what follows, will be $L^0 (\wp, H)$, the $F$-space of all $H$-valued random variables with the topology of convergence in measure and $L^2(\wp, H)$, the Hilbert space of all second order $H$-valued random variables with the natural scalar product
\begin{equation}\label{eq:prod.scal.L2}
\scal{f}{g}_{L^2(\wp , H)} :\ = \int_\Omega \scal{f(\omega)}{g(\omega)}_H {\mathrm d} \wp (\omega) ,\qquad f, g \in L^2(\wp, H).
\end{equation}
Now we can define the fundamental concept of the paper, followed by a brief discussion.
\begin{defn}
When $\Lambda$ is an arbitrary set, then a mapping $\Phi$
\begin{equation} \label{eq:stoch.mapping}
\Lambda \ni \lambda \mapsto \Phi (\lambda) \in L^0(\wp , H) \ \  (\text{or }\  L^2 (\wp , H))
\end{equation}
is called a \emph{multivariate (second order) stochastic mapping} or briefly {\it m.(s.o.)s.m.} of index set $\Lambda$.

Is $\Lambda$ endowed with a structure of linear space, separate topological space or measurable space, then it is usually required for $\Phi$ to be also \emph{linear} and/or \emph{continuous} and/or \emph{measurable}, respectively.
\end{defn}

For a study of such m.(s.o.)s.m. see \cite{ThangRndMapInfDimSp98}, \cite{ThangSerRepr98}, \cite{GaPopa}, \cite{GaPopaII}, \cite{GaPopaIII}.

In this paper we restrain ourselves to the case of (s.o.) random operators, which are defined as follows.
\begin{defn}[see \cite{Hack4}, \cite{Hack5}] \label{def:rnd.op}
A multivariate (second order) stochastic linear mapping $A$, whose index set is a dense linear subspace $D(A)$ in some complex separable Hilbert space $G$, is called a \emph{(second order) random operator from $G$ to $H$} (or simply on $H$ if $G = H$).

When $D(A) = G$ and $A$ is linear and continuous then it is called a \emph{continuous (s.o.) random operator}, or, as in \cite{Hack5}, \emph{(s.o.) Skorohod operator.}
\end{defn}

Are the values of a (s.o.) random operator constant functions from $L^0(\wp, H)$ (or from $L^2(\wp, H)$), then we have an ordinary densely defined operator from $G$ to $H$. Hence, the (s.o.) random operators from $G$ to $H$ (respectively the continuous (s.o.) random operators from $G$ to $H$) are stochastic generalizations of deterministic densely defined (respectively continuous) linear operators from $G$ to $H$. They are needed in modelling some phenomena, where random outputs rather than deterministic ones can be identified.

The class $\Lc_d(G, L^p (\wp , H))$ of (s.o.) random operators from $G$ to $H$ (for $G=H$ see \cite{Hack4}, \cite{Hack5}) will be henceforth denoted by $\Rs^p (\Omega; G, H)$, while the class $\Bc (G, L^p (\wp , H))$ of continuous (s.o.) random operators (first studied in \cite{Sko}; see also \cite{ThangRndOpBan87}) will be denoted here by $\Ss^p(\Omega ; G, H)$ ($\subset \Rs^p(\Omega; G,H)$) for $p=0$ and $p=2$, respectively\footnote{Let us mention that in \cite{ThangRndOptoRndBndOp08} the class $\Ss^p (\Omega; G, H)$ of \emph{$p$-order random operators} for each $p\in [0, \infty)$ was considered. Naturally, it is possible to introduce also the class $\Rs^p (\Omega; G, H) (\supset \Ss^p(\Omega; G, H))$ of \emph{$p$-order densely defined random operators} from $G$ to $H$, when in \eqref{eq:stoch.mapping} we consider the Lebesgue spaces $L^p (\wp, H)$.}.

After A. V. Skorohod (see for example \cite{Sko}) the class of continuous random operators was also studied (for $p=0$) by D. H. Thang and others (\cite{ThangRndOpBan87} -- \cite{Thang4}),
who often replace the Hilbert spaces $G$ and $H$ by Banach spaces $X$ and $Y$, respectively.


Now a class of general random operators can be constructed as follows (see \cite{Hack4}, \cite{Hack5}).
\begin{ex}\label{exmp:unu}
Let $\{ a(\omega)\}_{\omega \in \Omega}$ be a family of deterministic densely defined closed linear operators $a(\omega) \ (\text{i.e.,} \in \Cc\Lc^d (G, H)$), indexed by the random parameter $\omega \in \Omega$, which is measurable (2 - summable) in the following sense: it admits a common core $D$ ($D \subset D(a(\omega))\subset G$) such that for each $x \in D,\ A x \in L^0(\wp , H)$ ($L^2(\wp, H)$,  respectively), where
\begin{equation} \label{eq:def.rnd.op}
(A x)(\omega) :\ = a(\omega) x , \text{ for almost all } \omega \in \Omega.
\end{equation}
Then we have that $A\in \Rs^p (\Omega ; G, H),\  p = 0, 2$ ($D = D(A)$ being dense in $G$). This operator will be called the \emph{(second order) random operator from $G$ to $H$ associated to the measurable (2 - summable) family $\{ a(\omega)\}_{\omega \in \Omega}$}.

\end{ex}

When we consider a measurable (2 - summable) family $ \{ a(\omega)\}_{\omega\in\Omega}$ of deterministic continuous linear operators from $G$ to $H$, the above measurability (2-summability) condition is naturally satisfied for $D = G$. So $A$ defined as in \eqref{eq:def.rnd.op} is a \emph{continuous (s.o.) random operator} from $G$ to $H$. This subclass $\Ls^p (\Omega; G, H)$ of $\Ss^p(\Omega; G , H)$ was introduced in
\cite[Section 4]{ThangRndOpBan87} and extensively treated for $p=0$ in \cite{ThangRndBndOpKyu98}, where its elements were characterized in terms of the boundedness condition (Theorem 3.1) and also in terms of its decomposable extension (Theorems 5.2 and 5.3). Resuming we can formulate
\begin{prop}\label{prop:bnd.cond}
Let $A\in \Ss^p(\Omega; G, H)$ be a (s.o.) Skorohod operator from $G$ to $H$. Then the following conditions are equivalent:
\begin{enumerate}[(i)]
\item $ A \in \Ls^p(\Omega; G, H)$;
\item There exists a positive random variable $r(\cdot) \in L^p(\wp) :\ = L^p (\wp, \C)$ such that
\begin{equation}\label{eq:rnd.bnd.cond}
\| Ax (\omega) \| \le r(\omega) \|x\|,\quad x\in G,\quad \omega \text{ - a.e. }
\end{equation}
\item $A$ admits a continuous linear extension $\tilde A$ from $G$ to the whole $L^p (\wp, G)$ such that
\begin{equation}\label{eq:ext}
\tilde A (\chi_\alpha x) = \chi_\alpha A x ;\quad x\in G, \  \alpha \in \As,
\end{equation}
$\chi_\alpha$ being the characteristic function of the set $\alpha$.
\end{enumerate}
\end{prop}
Let us also note the following properties of the class $\Ls^p (\Omega ; G, H)$.
\begin{rem}\label{rem:extens}
When in \eqref{eq:rnd.bnd.cond} above $r(\cdot)\in L^2(\wp)$, then $A \in \Ls^2 (\Omega ; G, H)$, while the relation \eqref{eq:ext} implies
\begin{equation}\label{eq:ext.gen}
  \tilde A (\varphi x) = \varphi A x,\quad x\in G, \ \varphi \in L^q (\wp),
\end{equation}
where $q=0$ or $q=\infty$ if $p=0$ or $p=2$ respectively. The first statement is a simple verification , while \eqref{eq:ext.gen} results by taking the limit in \eqref{eq:ext} (for $p=q=0$ see property (2) \cite[p. 271]{ThangRndBndOpKyu98}).
\end{rem}

\begin{rem}\label{rem:doi}
The linear subspace $\Ls^p (\Omega ; G, H)$ is dense in $\Ss^p (\Omega; G, H)$ in the topology of pointwise convergence of operators. The proof given for $p=0$ in \cite[Remark (ii) Section 1]{Hack4} goes also for $p=2$ with a slight modification.

Now, by using a version of a ``nuclear spectral theorem''
(\cite[Proposition 12.2.1, pp. 336]{Schmud}) we obtain that the Hilbert Schmidt operators from
$G$ to $L^2 (\wp , H)$ are decomposable s.o. Skorohod operators from $G$ to $H$.
\end{rem}
More precisely, we have
\begin{rem} \label{rem:trei}
The operator space $\Cs_2 (G, L^2(\wp, H))$ is contained not only in $\Bc (G, L^2 (\wp , H)) = \Ss^2 (\Omega ; G, H)$ but also in $\Ls^2 (\Omega ; G, H)$:
\begin{equation} \label{eq:inclus.Hilb.Schm}
\Cs_2 (G, L^2 (\wp ,H)) \subset \Ls^2 (\Omega ; G, H) \subset \Ls^0 (\Omega ; G, H) ,
\end{equation}
which means that, for each $A \in \Cs_2 (G, L^2 (\wp , H))$ and each $\omega \in \Omega$, there exists a Hilbert Schmidt operator $a(\omega) \in \Cs_2(G, H)$ such that for each $x\in G,\ (Ax)(\omega) = a(\omega) x$, $\omega$ - a.e., where $a(\cdot) x \in L^2(\wp, H)$, what means $A \in \Ls^2 (\Omega ; G, H)$.
Conversely, if $A$ is a second order random operator from $G$ to $H$, associated to a family $\{a(\omega)\}_{\omega \in \Omega}$ for which $a(\omega) \in \Cs_2 (G, H), \omega \in \Omega$, then $A\in \Cs_2 (G , L^2 (\wp , H))$.
\end{rem}

\section{Decomposable extensions and adjoints of general random operators}
Now, having in view the three above equivalent conditions from Proposition \ref{prop:bnd.cond}, as well as the results from \cite{Hack4}, \cite{Hack5}, in order to identify the analogue of the class $\Ls^p(\Omega; G, H)$ in the unbounded framework, the most adequate way would be to transpose in this setting and discuss first condition (iii) regarding the special extension , which because of \eqref{eq:ext.gen}, as we shall see, will be a decomposable operator. To this aim let us begin by recalling some facts about pointwise action operators between $L^p (\wp , G)$ and $L^p (\wp , H)$.
\begin{defn}[see \cite{Hack4}]\label{def:op.pointw.appl}
The mapping $\bA_{\mathbf a} : D (\bA_{\mathbf a}) \to L^p (\wp , H),\ (p=0 \text{ or } 2)$, where ${\mathbf a} = \{ a(\omega) \}_{\omega \in \Omega}$ is a measurable ($2$ - summable) family of operators from $G$ to $H$, defined by
\begin{equation} \label{eq:def.pointw.op}
\bA_{\mathbf a} f (\omega) :\ = a(\omega) f(\omega), \qquad f \in D(\bA_{\mathbf a}),\text{ for almost all } \omega\in\Omega
\end{equation}
where
\begin{multline} \label{eq:dom.pointw.op}
D(\bA_{\mathbf a}) :\ = \{ f \in L^p(\wp , G) : \ f(\omega) \in D(a(\omega)) \text{ a.e.}, \\
\text{ such that } \ a(\cdot) f(\cdot) \in L^p (\wp , H) \},
\end{multline}
will be called the \emph{operator of pointwise application} associated to the family ${\mathbf a}$.
\end{defn}
An important example of an operator of pointwise application is the so called \emph{diagonalizable operator $m_\varphi^G$} \footnote{$m^G_\varphi f = \varphi f, f \in L^p (\wp, G)$, the multiplication operator with the scalar measurable function $\varphi\in L^q(\wp)$.},i.e. that which is associated to the family of operators $\{ \varphi(\omega) 1_G \}_{\omega\in\Omega}$, with $\varphi \in L^q (\wp) = L^q (\wp , {\mathbb C})$, where here and throughout the paper, to $p=0$ and $p=2$ correspond $q=0$ and $q=\infty$, respectively.
\begin{defn}\label{def:decompos.closed.op}
A closed operator $\bA : D(\bA) (\subset L^p(\wp, G)) \to L^p(\wp, H),\ p=0,\ 2$ is decomposable, if it intertwines all pairs of diagonalizable operators on $G$ and on $H$ (i.e., $\bA  m_\varphi^G \supset m^H_\varphi \bA ,\ \varphi \in L^q (\wp),\ q=0$, respectively $q=\infty$).
\end{defn}
Let us now remark that the measurable families $\ba = \{a(\omega)\}_{\omega \in \Omega}$, in terms of which the operators $A\in \Ls^p (\Omega; G, H)$ are described as in Proposition \ref{prop:bnd.cond}, are basically the same as the measurable fields of continuous linear operators from $G$ to $H$, used in the book \cite{Dixm}. Along this idea, as was observed in \cite{Hack4} and \cite{Hack5} the most appropriate measurable families which permit to treat the decomposability of random unbounded operators in the spirit of Example \ref{exmp:unu} are the measurable fields of unbounded linear operators used in \cite{Schmud}, which are defined as follows.
\begin{defn}\label{def:msrb.fld}
 A family ${\mathbf a} = \{ a(\omega)\}_{\omega \in \Omega}$ of closed linear operators from $G$ to $H$ is called a \emph{measurable ( 2 - summable) field} of operators, when there exists a sequence $f_n \in L^0 (\wp, G)$ ($f_n \in L^2 (\wp, G)$, respectively), such that for each $\omega \in \Omega,\ f_n (\omega) \in D(a(\omega))$ and $\{ (f_n(\omega), g_n(\omega))\}_{n\in \N}$ is a total subset of the graph ${\mathcal G}_{a(\omega)}$ of $a(\omega)$, where $g_n (\omega) = a(\omega) f_n(\omega)$.
\end{defn}
As is easily seen, a measurable (2 - summable) family of continuous linear operators $\{ a(\omega) \}_{\omega\in\Omega}$ from $G$ to $H$ is a measurable (2 - summable) field of operators. Also, by using the characterizations from section 12.1 of \cite{Schmud} it is not difficult to see that the families used in
Example \ref{exmp:unu} fit in this category.


Now the definitions \ref{def:op.pointw.appl}, \ref{def:decompos.closed.op}, \ref{def:msrb.fld}, are connected through (see \cite{Schmud}, \cite{Hack5} for $p=2$ and \cite{ThangRndBndOpKyu98} for $p=0$)

\begin{prop} \label{prop:decompos.clsd.op}
A closed operator $\bA$ from $L^p (\wp , G)$ into $L^p(\wp, H),\ p=0,2$ is decomposable, iff it is an operator of pointwise application  associated to some (uniquely determined) measurable ($2$ - summable) field ${\mathbf a} = \{ a(\omega) \}_{\omega \in \Omega}$ of operators from $G$ into $H$, i.e. $\bA = \bA_{{\mathbf a}}$.

If this is the case we also say that $\bA$ is \emph{decomposable by the field ${\mathbf a}$}.
\end{prop}

It is the place now to remark that there are circumstances in which it is possible that not only the outputs but also the inputs are under the influence of a stochastic environment. It makes then perfect sense to consider extensions of (second order) random operators on $G$, to densely defined operators on $L^0 (\wp, G)$ (respectively on $L^2 (\wp, G)$), or just of random operators from $G$ to $H$ to densely defined operators from $L^2(\wp , G)$ to $L^2 (\wp , H)$. But such extensions are useful to our development when they are also decomposable. Therefore we give (see also \cite{Hack4} or \cite{Hack5})
\begin{defn}\label{def:extension}
  An operator $\bA \in \Lc_d (L^p(\wp , G), L^p (\wp , H))$ (for $p=0$ and $p=2$ respectively) is a \emph{decomposable extension} of the (s.o.) random operator $A$ from $G$ to $H$ (i.e. of $A \in \Rs^p (\Omega ; G, H)$), when for $q=0$ and $q=\infty$ respectively, we have
  $$D(\bA) \supset L^q(\wp) \otimes D(A) \supset D(A),$$ where
  \[ L^q(\wp) \otimes D(A) = \left\{ \sum\limits_{j=1}^{n} \varphi_j x_j, \varphi_j \in L^q (\wp), x_j \in D(A), n \in {\mathbb N} \right\}\]
  is a core of $\bA$ and
  \begin{equation}\label{eq:extension}
    \bA (\varphi x) = \varphi A x , \text{ for } \varphi \in L^q (\wp) ,\ x \in D(A).
  \end{equation}
\end{defn}

When $A$ is just a random operator from $G$ to $H$ (i.e. $A \in \Rs^0 (\Omega; G, H)$), then we also define a ``Hilbert type'' decomposable extension $\bA \in \Lc_d(L^2(\wp , G), L^2(\wp , H))$ through \eqref{eq:extension} but for $\varphi \in L^\infty (\wp)$ and $x \in D(A)$ satisfying $\varphi A x \in L^2 (\wp , H)$ and having a core
\begin{multline*}
 C(A) = \{ \sum_{j=1}^{n} \varphi_j x_j , \varphi_j \in L^\infty (\wp) ,\ x_j \in D(A)\ \text{ s.t. } \\
 \varphi_j A x_j \in L^2 (\wp  ,H), j=1, 2, \dots , n ,\ n \in \N \}.
\end{multline*}
The decomposable extension $\bA$ of $A$ is called \emph{minimal}, when $\bA \subset \bA'$ for each decomposable extension $\bA'$ of $A$.

\begin{rem} \label{rem:1.1}
The operator of pointwise application $\bA_{\mathbf a}$ associated to a $2$ - summable field ${\mathbf a}$ can serve as extension in the sense of Definition \ref{def:extension} of the second order random operator $A$, associated to the $2$ - summable family ${\mathbf a} = \{ a(\omega)\}_{\omega \in \Omega}$ as in Example \ref{exmp:unu}.

Indeed, it is obvious that $D(\bA_{\mathbf a})$ contains the elements from $D(A) (\subset G)$ as the constant functions from $L^2 (\wp , G)$ as well as the subspace $L^\infty (\wp) \otimes D(A)$, which easily leads to the fact that $\bA_{\mathbf a}$ extends $A$.
\end{rem}

For the existence of the minimal decomposable extension, first the notion of random adjoint is needed, concept introduced and treated for Skorohod operators in \cite[Lemma 3.3 and Theorem 3.5]{ThangAdjCompRndOp95} and \cite[Definition 3.6 and Lemma 2]{Thang4}, while for general random operators it is treated in \cite{Hack4} and \cite{Hack5}.


The concept of random adjoint, also necessary for the development of random spectral theory will be defined directly in the general case.

\begin{defn}\label{def:rand.adjoint}
For a (s.o.) random operator $C$ from $G$ to $H$,
i.e. $C \in \Rs^p (\Omega; G, H)$, the \emph{(s.o. if $p=2$) random adjoint} $C^\bullet$ will be defined by $C^\bullet y = f_y$, for $y \in D(C^\bullet) \subset H$, where
\begin{multline*}
D(C^\bullet):\ = \bigl\{ y \in H :\ \exists f_y \in L^p (\wp, G) \text{ s.t. for each } \ x \in D(C) \\ \scal{(Cx) (\omega)}{y}_H = \scal{x}{f_y(\omega)}_G \text{ a.e.} \bigr\}.
\end{multline*}
\end{defn}
Let us note that the separability of the Hilbert spaces $G$ and $H$ and the density of $D(C)$ in $G$ ensures that $C^\bullet$ is well defined.

The fundamental properties of the random adjoint which show some analogies with the natural adjoint are collected in
\begin{prop}\label{prop:propr.bullet}
If $A$ and $B$ are (s.o.) random operators from $G$ to $H$ then:
\begin{enumerate}[(i)]
\item $A^\bullet$ is a closed operator from $H$ into $L^p (\wp , G)$;
\item when $A^\bullet$ is a (s.o.) random operator from $H$ to $G$ (i.e. $D(A^\bullet)$ is dense in $H$) then \begin{itemize}
    \item[(ii.a)] $(A^\bullet)^\bullet =\ : A^{\bullet \bullet}$ makes sense, is again a (s.o.) random operator and extends $A$ ($A \subseteqq A^{\bullet \bullet}$);
    \item[(ii.b)]  $A$ is closable;
    \item[(ii.c)] $(\bar A)^\bullet = A^\bullet$ and $\bar A = A^{\bullet\bullet}$;
    \item[(ii.d)] $A$ is closed iff $A= A^{\bullet\bullet}$;
    \end{itemize}
\item $A\subseteqq B$ implies $B^\bullet \subseteqq A^\bullet$;
\item $(\lambda A )^\bullet = \bar \lambda A^\bullet , \ \lambda \in {\mathbb C}$ ;
\item when $A+B$ is a random operator (i.e. $ D(A+B) :\ = D(A) \cap D(B)$ is dense in $G$), then $(A+B)^\bullet \supseteqq A^\bullet + B^\bullet$;
\end{enumerate}
\end{prop}
\begin{proof}
  For (i) and (ii) see \cite{Hack4} or \cite{Hack5}.

  (iii). Let $y \in D(B^\bullet) (\subset H)$. Then, by using $A\subseteqq B$, the following implication holds:
  \begin{multline*}
    \langle (Bx)(\omega), y \rangle_H = \langle x ,(B^\bullet y) (\omega) \rangle_G ,\ x\in D(B),\ \omega - \text{a.e.} \ \Rightarrow \\
    \langle (Ax)(\omega), y \rangle_H = \langle x ,(B^\bullet y) (\omega) \rangle_G ,\ x \in D(A) ,\ \omega - \text{a.e.}
  \end{multline*}
  This leads to $y \in D(A^\bullet)$ and the density of $D(A)$ in $G$ implies $(B^\bullet y )(\omega) = (A^\bullet y) (\omega) \ \omega - \text{ a.e.}$, i.e. $B^\bullet \subseteqq A^\bullet$.

  (iv) results directly from the definition of the random adjoint.

  (v) Let $y\in D(A^\bullet + B^\bullet) :\ = D(A^\bullet) \cap D(B^\bullet)$. Then for each $x\in D(A+B)$ we have:
  \begin{align*}
    \langle (A+B)x (\omega) , y \rangle_H & = \langle Ax (\omega), y \rangle_H + \langle Bx (\omega), y \rangle_H \\
     &  = \langle x, (A^\bullet y) (\omega)\rangle_G + \langle x, (B^\bullet y) (\omega)\rangle_G \\
     & = \langle x, (A^\bullet + B^\bullet) y (\omega) \rangle_G , \ \omega - \text{ a.e.}
  \end{align*}
  which by the hypothesis of (v), implies that $y \in D((A+B)^\bullet)$ and $(A+B)^\bullet y = A^\bullet y + B^\bullet y$, from where $(A+B)^\bullet \supseteqq A^\bullet + B^\bullet.$
\end{proof}

We add now some simple facts of the random adjoint for the case of continuous random operators.

First it is not difficult to see, that the following results proven for $p=0$ in Theorem 3.10 of \cite{ThangAdjCompRndOp95} and Lemma 2 of \cite{Thang4}, hold for $p=2$ as well.

\begin{rem} \label{rem:Skorohod.adjoint}
Each continuous random operator from the class $\Ls^p(\Omega; G, H)$ posses a random adjoint.
Moreover, when $A_{\mathbf a} \in \Ls^p(\Omega; G, H)$ is the continuous random  operator associated to a family ${\mathbf a} = \{ a(\omega)\}_{\omega\in\Omega}$ from $\Bc(G, H)$ as in \eqref{eq:def.rnd.op}, then $A_{\mathbf a}^\bullet$ is a Skorohod operator from the class $\Ls^p (\Omega; H, G)$ and $A_{\mathbf a}^\bullet = A_{{\mathbf a}^*}$, where ${\mathbf a}^* = \{ a(\omega)^*\}_{\omega \in\Omega}$ is the family from $\Bc(H, G)$ of the corresponding classical adjoints.
\end{rem}

Having in view Remark \ref{rem:trei}, we are now led to

\begin{rem}
The Hilbert Schmidt subclass of s.o. Skrohod operators is ``invariant'' to the random adjoint, in the sense that
\begin{equation}\label{eq:invarianta.Hilb.Schm.la.punct}
A\in \Cs_2(G, L^2 (\wp, H)) \Rightarrow A^\bullet \in \Cs_2 (H , L^2 (\wp , G)).
\end{equation}
Moreover, by \ref{prop:propr.bullet} (ii.d) the inverse implication in \eqref{eq:invarianta.Hilb.Schm.la.punct} holds.
\end{rem}
The following example, important in itself, illustrates the difference between the random adjoint and the classical adjoint.

\begin{ex}\label{exmp:doi}
  The embedding operator of $H$ into $L^2 (\wp , H)$, denoted in the following by $J_H$ is a s.o. random operator, which is decomposable by the constant family $\{ 1_H \}_{\omega \in \Omega}$, where $1_H$ represents the identity operator on $H$. Therefore $J_H \in \Ls^2 (\Omega ; H)$. By Remark \ref{rem:Skorohod.adjoint} since $1_H^* = 1_H$, we have that $J_H^\bullet = J_H$. On the other side, trying to compute the classical adjoint of $J_H$ the following simple calculation, for all $y \in H$, and $g \in L^2 (\wp , H)$,
  \[ \langle y , J_H^* g \rangle_H = \langle J_H y , g \rangle_{L^2(\wp , H)} = \int \langle y , g(\omega) \rangle_H d\wp (\omega) = \langle y , \int g(\omega) d \wp (\omega)\rangle_H , \]
  leads to the fact that $J_H^* = {\mathbb E}_H$, where ${\mathbb E}_H$ is the expectation operator on $L^2 (\wp , H)$. Thus, 
  the random adjoint of the second order random operator $J_H$ (which is $J_H$) differs essentially from its classical adjoint (which is ${\mathbb E}_H$).
\end{ex}



Now, returning to the random adjoint in the unbounded framework, its role in the existence of decomposable extensions will be described in the following two theorems (see \cite{Schmud}, \cite{Hack4}, \cite{Hack5}).

\begin{thm} \label{thm:car.rnd.adj}
Let $A \in \Rs^0 (\Omega; G, H)$ be such that $A^\bullet \in \Rs^0 (\Omega; H, G)$. Then there is a uniquely determined minimal decomposable extension of ```Hilbert type'' $\bA$ \footnote{Actually $\bA$ is the closure of the restriction to the domain from \eqref{eq:comm.core} of the random adjoint of $A^\bullet$ regarded as second order random operator (see Lemma and Theorem 1 from \cite{Hack5}).} of $A$,
such that
    $\bA (\varphi x) = \varphi A x$, $\varphi \in L^\infty (\wp), x\in D(A) \text{ s.t. } \varphi Ax \in L^2(\wp, H) $.\\
    Moreover,
     \begin{multline}\label{eq:comm.core}
      C(\bA) = \Bigl\{ \sum\limits_{j=1}^{m} \varphi_j x_j ,\text{ with } \varphi_j \in L^\infty (\wp) \text{ and }  x_j \in D(A),\text{ s.t. } \\ \varphi_j A x_j \in L^2 (\wp, H),
       j=1,2,\dots , m ;\ m \in {\mathbb N}  \Bigr\}
     \end{multline}
      is a core for $\bA$.
\end{thm}

The properties of the random adjoint, of the decomposable extensions (see \cite{Hack5}) and of their connection can be reformulated as follows
\begin{thm}\label{thm:propr.rnd.op}
Let $A \in \Rs^0 (\Omega; G, H)$ be such that the restriction $A|_{D^2(A)} \in \Rs^2 (\Omega; G, H)$, where $D^2 (A) :\ = \{ x \in D(A), \ \text{ s.t. }\ A x \in L^2 (\wp, H)\}$.
Then the following assertions hold
\begin{itemize}
\item[(i)] $A^\bullet \in \Rs^0 (\Omega; H, G)$, $\bigl(A|_{D^2 (A)}\bigr)^\bullet \in \Rs^2(\Omega; H, G)$;
\item[(ii)] $A$ has a uniquely determined minimal decomposable extension $\bA = \bA_{\mathbf a}$ by a $2$-summable field ${\mathbf a} = \{ a(\omega) \}_{\omega \in \Omega}$ of unbounded operators from $G$ to $H$, for which $D^2(A)$ is a common core;
\item[(iii)] $\bA_{\mathbf a}$ and almost all $a(\omega), \omega \in\Omega$ are densely defined;
\item[(iv)] $\{ a(\omega)^*\}_{\omega \in\Omega} =\ : {\mathbf a}^*$ is a $2$ - summable field of unbounded operators from $H$ to $G$, $\bA_{{\mathbf a}^*} = (\bA_{\mathbf a})^*$ and it is a decomposable extension of $A^\bullet$, consequently it also extends the minimal decomposable extension of $A^\bullet$;
\item[(v)] when $G = H$ and $A$ is a symmetric random operator, then its decomposable extension $\bA_{\mathbf a}$ is a symmetric operator on $L^2(\wp, H)$ in the usual sense.
\end{itemize}
\end{thm}
\begin{rem} \label{rem:doi.patru}
It is obvious that the adjoint of an s.o. random operator is again an s.o. random operator and then the statements (ii) - (v) from above are fulfilled for such operators.
\end{rem}
Having in view the connections of (s.o.) random operators with its decomposable extensions (if they exist), in the unbounded framework the analogue of the above mentioned class $\Ls^p (\Omega ; G, H)$ will be defined as follows.
\begin{defn}\label{def:decompos.rnd.op}
A \emph{(s.o.) random operator} $A$ from $G$ to $H$ will be called \emph{decomposable} (denoted $A \in \Ds^p (\Omega; G, H)$) 
, if there exists a measurable ($2$-summable) field ${\mathbf a} = \{a(\omega)\}_{\omega \in \Omega},\ a(\omega) \in \Cc\Lc (G, H)$, such that the operator of pointwise application  $\bA_{{\mathbf a}}$ associated to the family ${\mathbf a}$ is the minimal decomposable extension of $A$.
\end{defn}
In what follows, if no danger of confusion exists,
we say in this case that $A$ is \emph{decomposable by the measurable ($2$-summable) field ${\mathbf a}$} and denote that by $A = A_{\mathbf a}$\footnote{Sometimes it is said that the field $\{a(\omega)\}_{\omega \in \Omega}$ acts as the (s.o.) random operator $A_{\mathbf a}$ by the above pointwise application \eqref{eq:def.rnd.op} (see \cite{Hack4} Sec. 1 and \cite{Hack5} Sec. 2)}.
\begin{rem}\label{rem:2.5}
When the (s.o.) random operator $A$ from $G$ to $H$ is decomposable by the (s.o.) measurable field ${\mathbf a}$, then it can be expressed in the pointwise form
\[
(Ax)(\omega) = a(\omega) x,\quad x\in D(A), \omega \in \Omega.
\]
Indeed, it is obvious that when $x\in D(A)$, then $x= 1\otimes x \in L^q (\wp) \otimes D(A) \subset D(\bA_{{\mathbf a}})$, which by Definition \ref{def:extension} means $1 Ax =  \bA_{{\mathbf a}} x$, from where the desired relation holds.
\end{rem}

Remarking that the decomposable random operators are closed let us mention the
following inclusion diagram
\begin{equation*}
\begin{array}{ccc}
  \Rs^p (\Omega;G,H) &\supset& \Ss^p(\Omega;G,H) \\
  \cup & & \cup \\
  \Ds^p (\Omega;G,H) &\supset & \Ls^p(\Omega;G,H).
\end{array}
\end{equation*}

\section{Random normal operators}

Random selfadjointness and even random normality as well as the random spectral theorems were treated in \cite{Thang4}, but only in the continuous framework.

For the moment we give the definitions of selfadjoint random operators (for which we consider of course $G = H$) (see \cite{Hack4}, \cite{Hack5}).

\begin{defn}
When $C$ is a (s.o.) random operator and $C^\bullet$ extends $C$, we shall say that $C$ is a \emph{symmetric (s.o.) random operator} on $H$. In this case $C^\bullet$ is also densely defined and consequently a (s.o.) random operator. When, moreover $C = C^\bullet$, then $C$ will be called a \emph{self-adjoint (s.o.) random operator} on $H$.
\end{defn}




Further, in order to introduce the random normality it is necessary to define a composition \footnote{Another way to define the composition of a pair of (continuous) random operators can be found in \cite{ThangAdjCompRndOp95} Section 4.} for the decomposable random operators (i.e. from the classes $\Ds^p, \ p = 0, 2$). Namely

\begin{defn}
Let $A \in \Ds^p (\Omega; G, H)$, $B \in \Ds^p (\Omega; H, K)$ 
and $\bB$ the decomposable extension of $B$. 
Then the composition $BA$ will be defined by
\begin{equation}\label{eq:def.comp.rnd.op}
(BA) (x) :\ = \bB (Ax) , \text{ for } x \in D(BA),
\end{equation}
where $D (BA) = \{ x \in D(A) \text{ such that } Ax \in D(\bB)\}$.
\end{defn}


For the case of continuous random operators we immediately have
\begin{rem} \label{rem:Skor.compos}
If $A \in \Ls^p (\Omega; G, H)$ and $B \in \Ls^p(\Omega; H, K)$, the definition \eqref{eq:def.comp.rnd.op} works without restriction for any $x \in G$. Also $B A \in \Ls^p (\Omega; G, K),\ p = 0, 2$. Moreover, when $A$ and $B$ are decomposable by the measurable ($2$-summable) fields of operators $\{ a(\omega) \}_{\omega\in\Omega} ( \subset \Bc(G, H))$ and $\{ b(\omega) \}_{\omega\in\Omega} ( \subset \Bc(H, K))$, respectively, then the product $BA$ is the continuous random operator decomposable by the measurable ($2$-summable) field $\{ b(\omega) a(\omega) \}_{\omega \in \Omega} ( \subset \Bc(G, K)) $ (see also \cite{ThangAdjCompRndOp95}, \cite{Thang4}).
\end{rem}

\begin{rem}
We point out that, when $G=H=K$, then we can conclude that the linear space $\Ls^p (\Omega;H)$ endowed with the above random composition of operators as multiplication, with the random adjoint as involution and with the topology of pointwise convergence of operators is a topological $*$-algebra contained in $\Bc(H, L^p(\wp, H))$.
\end{rem}

For the case of general random operators it also holds
\begin{thm} \label{thm:comp.rnd.op.gen}
Let $A\in \Ds^p(\Omega; G,H),\ B\in\Ds^p(\Omega; H,K)$ be such that the two measurable ($2$-summable) fields of operators ${\mathbf a} = \{ a(\omega)\}_{\omega\in\Omega}$ and ${\mathbf b} = \{ b(\omega)\}_{\omega\in\Omega}$ to which $A$, respectively $B$ are associated (i.e. $A=A_{\mathbf a},\ B = B_{\mathbf b}$) satisfy the condition that $b(\omega) a(\omega) \in \Cc\Lc_d (G, K)$ for almost all $\omega \in \Omega$. Then $BA \in \Ds^p (\Omega; G, K)$ ($p= 0,2$) and it is associated to the measurable (2 - summable) family $\{b(\omega) a(\omega)\}_{\omega \in \Omega}$.
\end{thm}

\begin{rem}
Let us observe that the imposed condition on the product $b(\omega) a(\omega)$ to be in $\Cc\Lc_d (G, K)\ (\omega \in \Omega)$ can be characterized trough results from \cite{Len} (Proposition 2.2 (iv) and Proposition 4.1) given in terms of the characteristic matrices of $a(\omega)$ and $b(\omega)$ and of the bicharacteristic matrix for the product $b(\omega) a(\omega)$ (for recent studies on this topic see also \cite{AzouzMesDjell}, \cite{GustMortad}).
\end{rem}
\begin{proof}{\it of Theorem \ref{thm:comp.rnd.op.gen}}. (compare with Theorem 6.3 from \cite{Len}).
Let $x \in D(BA)$. Then $Ax \in D(\bB)$ and for almost all $ \omega \in \Omega$ we have
\[ (\bB A ) (x) (\omega) = (\bB (Ax))(\omega) = b(\omega) (Ax)(\omega) = b(\omega) a(\omega) x. \]
On the other side by hypothesis $c(\omega) = b(\omega) a(\omega)$ defines a measurable (2 - summable) field of operators from $G$ to $K$, which gives us that $BA$ is the (s.o.) random operator from $G$ to $K$ associated to ${\mathbf c} = \{ c(\omega)\}_{\omega\in\Omega}$.
\end{proof}
As in the case of ordinary adjoint (see also Proposition 1.7 of \cite{Schmud2}) we can easily complete the properties of random adjoint .
\begin{prop} \label{prop:comp.rnd.adj}
Let $A$ and $B$ be two (s.o.) random decomposable operators from $G$ to $H$, respectively from $H$ to $K$ such that $D(BA)$ is dense in $G$ and $D(B)$ is dense in $H$. Then $(BA)^\bullet \supseteqq  A^\bullet B^\bullet$. If $B$ is continuous then $(BA)^\bullet = A^\bullet B^\bullet$.
\end{prop}
Now, having defined the random operator composition we can (returning to the case of Skorohod operators) define the random projection operators.
\begin{defn}\label{def:rand.proj.op}
A random operator $P$ on the Hilbert space $H$ is called random projection operator, when it is continuous, decomposable, selfadjoint and idempotent (i.e. $P\in \Ls^0 (\Omega; H)$, $P = P^\bullet,\ PP = P$).
\end{defn}
\begin{rem}
  Having in view that for such random operators $P$ there is a family ${\mathbf p} :\ = \{ p(\omega)\}_{\omega \in \Omega}$ of  projection operators on $H$ such that $P = A_{{\mathbf p}}$ it results that $P$ being a contractive random operator (i.e. $\|(Px)(\omega)\| \le \|x\|,\ \omega - \text{a.e.} $) it is automatically a s.o. random operator, i.e. $P \in \Ls^2 (\Omega ; H)$.
\end{rem}

Two very simple continuous random operators on $H$ are to be mentioned here as examples of random projection operators: the null operator $O_H$ and the embedding operator $J_H$ from $H$ into $L^2 (\wp, H)$. It is simple to observe that its decomposable extensions ${\mathbf O}_H$ and ${\mathbf J}_H$ are the null operator $O_{L^2(\wp , H)}$, respectively the identity operator $I_{L^2(\wp , H)}$ on the Hilbert space $L^2(\wp , H)$. As discussed in Example \ref{exmp:doi} for $J_H$, they are decomposable by the constant field consisting of the null operator $O_H$, respectively of the identity operator $1_H$ on $H$. It is then obvious that $O_H^\bullet = O_H$ and $J_H^\bullet = J_H$, thus both being selfadjoint s.o. random operators on $H$ and consequently both are random projection operators on $H$.

The following basic properties of random projection operators (which are not difficult to verify) are stated in
\begin{prop}
Let $P$ and $Q$ be random projection operators on $H$. Then:
\begin{enumerate}[(i)]
\item When $PQ = QP$, then $PQ$ is a random projection operator as well.
\item  $PQ = O_H$, implies $QP = O_H$ and $P+Q$ is also a random projection operator. They are called orthogonal random projection operators. In particular $P$ and $J_H - P$ are orthogonal random projection operators.
\item If $P$ and $Q$ simply commute, then $P+Q - PQ$ is a random projection operator as well.
\item If $PQ = Q$, then $QP = Q$ and $P - Q$ is a random projection operator. Moreover $PQ= Q$ is equivalent to $P \ge Q$.
\item For any random projection operator $P$ it holds $P \le J_H$.
\end{enumerate}
\end{prop}

We now define the normal (s.o.) random operators in our general setting


\begin{defn}
A (s.o.) random operator $N \in \Rs^p(\Omega; H)$ ($p=0,2$) is said to be a normal (s.o.) random operator on $H$, if $N$ is closed 
$N \in \Ds^p(\Omega; H)$ and $NN^\bullet = N^\bullet N$.
\end{defn}
It is now not difficult to see that by applying Theorem \ref{thm:propr.rnd.op}, each selfadjoint second order random operator is a normal s.o. random operator.

On the other hand the following characterization shows that the above definition of random normality is equivalent in the case of continuous random operators to that given in \cite{Thang4}.
\begin{thm}\label{thm:car.normal}
A (s.o.) random operator $A$ on $H$ is normal, if and only if there is a ($2$-summable) measurable field ${\mathbf a} = \{a(\omega)\}_{\omega \in \Omega}$ of normal operators on $H$, such that $A = A_{\mathbf a}$.
\end{thm}
\begin{proof}
For the ``if'' part $D(A) = D(A^\bullet)$ implies that $A$ is densely defined, meaning that $A^\bullet \in \Rs^p(\Omega; H)$. Then, since $A \in \Ds^p(\Omega;H)$ it results that there is a measurable ($2$ - summable) field ${\mathbf a} = \{a(\omega)\}_{\omega \in \Omega}$ 
such that $A$ is random decomposable by ${\mathbf a}$, i.e. $A = A_{\mathbf a}$. Now, from Theorem \ref{thm:car.rnd.adj} and Theorem \ref{thm:propr.rnd.op} we have that $A^\bullet = A^\bullet_{\mathbf a} = A_{{\mathbf a}^*}$, i.e. $A^\bullet \in \Ds^p (\Omega;H)$. Moreover, ${\mathbf a}^* = \{a(\omega)^*\}_{\omega \in \Omega}$ has $D(A^\bullet)$ as common core and $D(A) = D(A^\bullet)$ means that ${\mathbf a}$ and ${\mathbf a}^*$ have the same common core.

Now, considering the extensions $\bA_{\mathbf a}$ and $\bA_{{\mathbf a}^*}$ of $A$ and $A^\bullet$ respectively, we see that $AA^\bullet = \bA_{\mathbf a} A_{{\mathbf a}^*} = A_{{\mathbf b}}$, where ${\mathbf b}=\{a(\omega) a(\omega)^*\}_{\omega \in \Omega}$ and $A^\bullet A = \bA_{{\mathbf a}^*} A_{\mathbf a}  = A_{{\mathbf c}}$, where ${\mathbf c} = \{a(\omega)^* a(\omega)\}_{\omega \in \Omega}$. Thus ${\mathbf a}{\mathbf a}^* = {\mathbf a}^*{\mathbf a}$, which together with the fact that the families ${\mathbf a}$ and ${\mathbf a}^*$ share a common core, implies that ${\mathbf a}$ consists of normal operators.

The ``only if'' part results easier through a straightforward reasoning.
\end{proof}
Now, it is adequate to introduce in the random operators framework the analogue of the spectral measure (i.e. projection operator valued measure) from classical operator theory.
\begin{defn}
  Let $H$ be a complex separable Hilbert space and $(\Gamma , \Sigma)$ a measurable space (i.e. $\Sigma$ is a $\sigma$-algebra of subsets of $\Gamma$). By a random projection operator valued measure (briefly, r.p.o.v. measure) on $(\Gamma , \Sigma , H)$ we mean a mapping $E$ from $\Sigma$ into $\Ls^p (\Omega; H)$ such that
        \begin{enumerate}[(i)]
          \item $E(\sigma)$ is a random projection operator on $H$ for each $\sigma\in \Sigma$;
          \item If $\sigma_j \in \Sigma,\ j \in {\mathbb N}$ are pairwise disjoint, then
          \[ E( \bigcup\limits_{i=1}^\infty \sigma_j ) x (\omega) = \sum\limits_{j=1}^\infty E(\sigma_j) x (\omega),\ \text{ for each } \ x \in H, \text{ and almost each } \omega ;\]
          \item $E(\Gamma) = J_H$.
        \end{enumerate}
\end{defn}
\begin{rem}\label{rem:add.propr.rnd.ms}
A r.p.o.v. measure has also the following two properties

\begin{enumerate}[(iv)]
\item $E(\varnothing) = O_H$.
\item $E(\sigma)E(\tau) = E(\sigma \cap \tau)$, for each pair $\sigma , \tau \in \Sigma$.
\end{enumerate}
Indeed, first when $\sigma_j = \varnothing,\ j\in {\mathbb N}$ in the above condition (ii), we infer (iv).

Now, considering in (ii) $\sigma_j = \varnothing$ for $j>k$, we see that (iv) and (ii) imply
\begin{equation*}
  E(\bigcup\limits_{j=1}^{k} \sigma_j) = \sum\limits_{j=1}^k E(\sigma_j).
\end{equation*}
Further, the finite additivity of $E$ implies the multiplicativity of $E$ in the following way (a similar technique as in \cite[Lemma 4.3, pp. 67]{Schmud2}).

The successive implications
\begin{multline*}
(J_H + E(\tau))x = 0 \Rightarrow \langle (J_H + E (\tau))x (\omega) , x \rangle = \|x\|^2 + \|E(\tau) x(\omega)\|^2 = 0 \\
\Rightarrow x=0
\end{multline*}
yield that $J_H + E (\tau)$ is injective for each $\tau \in \Sigma$. Now, for disjoint $\sigma, \tau  \in \Sigma$ the finite additivity of $E$ implies that $E(\sigma) + E (\tau)$, being equal to $E(\sigma \cup \tau)$, is again a random projection operator. Therefore $E(\sigma) + E(\tau) = (E(\sigma) + E(\tau))^2 = E(\sigma) + E(\tau) + E(\sigma)E(\tau) + E(\tau)E(\sigma)$ leads to $E(\sigma) E(\tau) + E(\tau) E(\sigma) = O_H $, which after multiplication to the right with $E(\tau)$ gives us successively $E(\sigma) E(\tau) + E(\tau) E(\sigma) E(\tau) = (J_H + E(\tau))E(\sigma) E(\tau) = O_H$, from where by the injectivity of $J_H + E(\tau)$ we have $E(\sigma) E(\tau) = O_H$. Further, for arbitrary $\sigma , \tau \in \Sigma$, denoting $\sigma_0 :\ = \sigma \cap \tau$, $\sigma_1 :\ = \sigma \setminus \sigma_0$ and $\sigma_2 :\ = \tau \setminus \sigma_0$, we see that all pairs $(\sigma_j, \sigma_k),\ j\ne k,\ j,k = 0,1,2 $ consisting of disjoint sets, by the second part above satisfy $E(\sigma_j) E(\sigma_k) = O_H$ for $j\ne k;\ j,k = 0,1,2 $. Applying this in the following calculus
\begin{multline*}
  E(\sigma) E(\tau) = E(\sigma_0 \cup \sigma_1)E(\sigma_0\cup \sigma_2) = (E(\sigma_0) + E(\sigma_1))(E(\sigma_0) + E(\sigma_2) ) = \\
  E(\sigma_0) + \sum\limits_{\substack{j,k = 0 \\ j\ne k} }^{2} E(\sigma_j) E(\sigma_k) = E(\sigma_0) = E (\sigma\cap \tau),
\end{multline*}
we get that (v) holds.
\end{rem}

Consequently, the above concept is in fact the same as the \emph{generalized random spectral measure} from \cite{Thang4}.

\begin{rem}\label{rem:trei.sapte}
For each pair $\sigma, \tau \in \Sigma$, we have that $E(\sigma)$ commutes with $E(\tau)$ with respect to the random composition.
Moreover, when $\sigma \supseteq \tau $, it is not hard to see that $E(\sigma) = E(\tau) + E (\sigma\setminus \tau) \ge E(\tau)$, the inequality being considered with respect to the order given by
\[
A\le B \Leftrightarrow \langle (Ax)(\omega), x\rangle \le \langle (Bx)(\omega) , x \rangle ,\ x\in H,\ \omega \- \ \text{a.e.}
\]
in the set of all continuous selfadjoint random operators.
\end{rem}


Let us recall now some basic facts about the integral with respect to the r.p.o.v. measure $E$ on $(\Gamma, \Sigma, H)$, which was defined in \cite{Thang4}. Namely, first for simple functions
\[ f(\gamma) = \sum\limits_{i=1}^n c_i \chi_{\sigma_i} (\gamma) ,\ \gamma\in\Gamma,\ c_i \in \C,\ \sigma_i \in \Sigma, \]
the integral $I_E(f) = I(f)$ is defined as
\[ (I(f)x)(\omega) = \sum\limits_{i=1}^n c_i E(\sigma_i) x (\omega) ,\ x \in H,\ \omega \in\Omega, \]
and then, by passing to the uniform limit of simple functions, $I(f)$, denoted by $\int\limits_\Gamma f(\gamma) d E (\gamma) $, becomes a continuous random normal operator on $H$, for each element $f$ of the Banach $*$-algebra $B(\Gamma, \Sigma)$, consisting of all bounded $\Sigma$ - measurable complex valued functions. In this way, the application $f \mapsto I(f)$ is a representation of the Banach $*$-algebra $B(\Gamma, \Sigma)$ into $\Ls^0 (\Omega; H)$. In fact it's properties are contained in the following (see Theorems 4.2 and 4.3 of \cite{Thang4})
\begin{thm} \label{thm:propr.int.op}
The mapping $f\mapsto I(f)$ has the following properties
\begin{enumerate}[(i)]
\item it is linear, i.e. for each $f, g \in B(\Gamma,\Sigma)$ and $\lambda \in \C$ it holds
\[ I(f+g) = I(f) + I(g) ;\quad I(\lambda f) = \lambda I(f) \]
\item it is multiplicative, i.e. for each $f, g \in B(\Gamma, \Sigma)$
\[ I(fg) = I(f) I(g) \]
\item it is involution preserving, i.e. for each $f \in B(\Gamma, \Sigma)$,
\[ I(f)^\bullet = I(\bar f), \]
where $\bar f$ stands for the complex conjugation of $f$;
\item $\lim\limits_n (I(f_n)x) = I (\lim\limits_n f_n) x$, for each $x\in H$ and any convergent sequence $\{ f_n\}_n $ from $B(\Gamma , \Sigma)$.
\end{enumerate}
\end{thm}
From the random projection operator valued measure $E$ it is possible to derive, as in the classical case, random (complex valued) measures on $(\Gamma, \Sigma)$ as follows.

If for $x,y \in H$ and $\omega \in \Omega$ we put
\begin{equation}\label{eq:spectr.ms}
\begin{split}
   E_{x,y}^\omega (\sigma)  & :\ = \langle E(\sigma)x(\omega), y \rangle_H ;\ \sigma\in\Sigma \\
      E_x^\omega (\sigma) & :\ =  E(\sigma) x(\omega) ,\ \sigma \in \Sigma,
\end{split}
\end{equation}
it is easy to see that $E_{x,y}^\omega$ is a complex measure, while $E_x^\omega$ is a $H$-valued measure on $(\Gamma, \Sigma)$. So, for each fixed pair $x,y \in H$, respectively each $x\in H$ we have that
$E_{x,y}:\ =\bigl(E_{x,y}^\omega \bigr)_{\omega \in \Omega}$ is a (positive if $x = y$) $L^0 (\wp)$ - valued measure, while $E_x:\ = \bigl(E_x^\omega\bigr)_{\omega \in \Omega}$ is a $L^0 (\wp, H)$ - valued measure on $(\Gamma, \Sigma)$. Moreover, looking at the $L^2(\wp, H)$ - valued measure, denoted $G_x$, from the proof of Theorem 4.3 (2) in \cite{Thang4}, we see that $G_x = E_x = \bigl( E_x^\omega\bigr)_{\omega \in \Omega}$, hence $E_x$ is, in fact, an $L^2(\wp, H)$ - valued measure on $(\Gamma, \Sigma)$. Finally it is a simple matter to observe that the measure $E_{x,y}$ is $L^2(\wp)$ - valued.

Moreover, as in the classical case (see \cite[Lemma 4.4]{Schmud2}) we have the following characterization
\begin{lem}
  A map $E$ of the $\sigma$-algebra $\Sigma$ into the set of random projection operators on $H$ is a r.p.o.v. measure if and only if it is multiplicative and, for each $x \in H$, the set function $E_{x,x}(\cdot)$ is a positive random measure.
\end{lem}
\begin{proof}
When $E$ is a r.p.o.v. measure, it satisfies (v) and, for each $x\in H$ it is obvious that $E_{x,x}$ is a positive random measure.

Conversely, if $E$ is a random projection operator valued mapping on $\Sigma$, such that $E_{x,x}$ is a $L^0 (\wp)$-valued measure for each $x\in H$, then, because of the finite additivity of all $E_{x,x},\ x\in H$, we have the finite additivity of $E(\cdot)$, which, in turn,  by the proof of Remark 3.5, implies $E(\varnothing) = O_H$. Considering $\{ \sigma_j\}_j $ a pairwise disjoint sequence of sets in $\Sigma$, we have $E(\sigma_j)E(\sigma_k) = O_H$ for $j \ne k$. This implies that $E(\sigma_j)x(\omega)$ and $E(\sigma_k) x(\omega)$ are pairwise orthogonal ($k\in \N , x\in H, \omega \in \Omega$). Now the countably additivity of $E_{x,x}$ yields
\[ E^\omega_{x,x} \Bigl(\bigcup\limits_{j=1}^\infty \sigma_j\Bigr) = \sum_{j=1}^{\infty} E_{x,x}^\omega (\sigma_j) = \sum_{j=1}^{\infty} \|E(\sigma_j)x(\omega)\|^2  \]
and by the pairwise orthogonality of $E(\sigma_j)x(\omega)$ and $E(\sigma_k) x(\omega)$ we conclude
\[ \sum_{j=1}^{\infty} E(\sigma_j) x(\omega) = E(\sigma) x(\omega) \quad \omega - \text{ a.e. } \qedhere\]
\end{proof}

By using now integration with respect to these measures and the properties from Theorem \ref{thm:propr.int.op} we can also infer other properties of the $*$-representation  $f \mapsto I(f)$ of $B(\Gamma, \Sigma)$ into $\Ls^2(\Omega; H)$ gathered in
\begin{prop}
For $f \in B(\Gamma, \Sigma)$ and $x,y \in H$ we have
  \begin{enumerate}[(i)]
    \item $I(f) x(\omega) = \int f(\gamma) d E_x^\omega (\gamma)$;
    \item $\langle I(f) x (\omega) , y \rangle_H = \int f(\gamma) d E_{x,y}^\omega (\gamma)$;
    \item $\|(I(f) x)(\omega) \|^2_H = \int |f(\gamma)|^2 d E_{x,x}^\omega (\gamma)$;
    \item $\|(I(f)x)(\omega)\| \le \|f\| \|x\|$.
  \end{enumerate}
The equalities above are true $\wp$ - a.e. Moreover, the function (of $\omega \in \Omega$) in (i) is in $L^2 (\wp, H)$, the one from (ii) is in $L^2(\wp)$ and the one from (iii) is in $L^1 (\wp)$.
\end{prop}
\begin{proof}
(i) follows from the definition of $I(f)$, first for simple functions and then for arbitrary $f$ from $B(\Gamma, \Sigma)$.

(ii) results by taking in (i) the scalar product with $y \in H$.

(iii) results by applying (ii) with $y = I(f) x(\omega)$ and then the multiplicativity and the invariance to involution of the mapping $f \mapsto I(f)$.

(iv) follows easily from (iii) since $E_{x,x}^\omega (\Gamma) = \|x\|^2,\ x\in H$.
\end{proof}
In this way we are led to the following
\begin{cor} \label{cor:int.norm.I}
The mapping $f\mapsto I(f)$ is a contractive random $*$-representation of $B(\Gamma, \Sigma)$ into $\Ls^2(\Omega;H)$.
\end{cor}
\begin{proof}
Integrating in (iii) with respect to the probability measure $\wp$, we have
\begin{align*}
\int \| I(f) x\|^2_{L^2(\wp,H)} & = \int_\Omega \Bigl(\int_\Gamma |f(\gamma)|^2 d E_{x,x}^\omega (\gamma)\Bigr) d \wp (\omega) \\
 & \le \|f\|^2_{B(\Gamma, \Sigma)} \int_\Omega\int_\Gamma d E_{x,x}^\omega (\gamma) d\wp (\omega) \\
 & = \|f\|^2_{B(\Gamma, \Sigma)} \int_\Omega E_{x,x}^\omega (\Gamma) d\wp(\omega) \\
 & = \|f\|^2_{B(\Gamma, \Sigma)} \int_\Omega \langle E(\Gamma) x(\omega), x\rangle d\wp (\omega) \\
 & = \|f\|^2_{B(\Gamma, \Sigma)} \int_\Omega \langle J_H (\omega)x, x\rangle d\wp (\omega) \\
 & = \|f\|^2 \int \|x\|^2 d\wp (\omega) = \|f\|^2_{B(\Gamma, \Sigma)} \|x\|^2,
\end{align*}
from where
\[
\|I(f) x\|_{L^2 (\wp, H)} \le \|f\|_{B(\Gamma, \Sigma)} \|x\| ;\quad x \in H,
\]
which leads to the conclusion that $I(f) x \in L^2(\wp, H)$, i.e. $I(f)$ is a s.o. random operator on $H$. Moreover, we finally have
\begin{equation} \label{eq:bnd.I(f)}
\|I(f)\|_{\Ls^2(\Omega; H)} \le \|f\|_{B(\Gamma, \Sigma)},\quad f \in B(\Gamma, \Sigma). \qedhere
\end{equation}
\end{proof}
Now we close this section by extending the integral $I_E (f)$ to unbounded $\Sigma$ - measurable functions.

For, let $\Mc = \Mc(\Gamma,\Sigma,E)$, be the $*$-algebra of $\Sigma$-measurable functions $f:\Gamma \to \C \cup \{\infty\}$, which are $E$-a.e. finite (i.e. $E(\{\gamma\in\Gamma:\ f(\gamma) = \infty \})=O_H$).


It is to be expected that the values of such $I(f)$ exceed $\Ls^2 (\Omega;H)$. This will be given by using a technique similar to that from subsection 4.3.2 of \cite{Schmud2}, where the concept of bounding sequence plays a principal role.

\begin{defn}
  Let $\Mc_1$ be a subset of $\Mc$. A sequence $\{\sigma_n\}_{n \in \N}$ from $\Sigma$ is a \emph{random bounding sequence} for $\Mc_1$, if each $f\in \Mc_1$ is bounded on $\sigma_n$, $\sigma_n \subseteq \sigma_{n+1},\ n \in \N$ and $E(\bigcup\limits_{n=1}^{\infty} \sigma_n) = J_H$.
\end{defn}
\begin{rem}
  \begin{enumerate}[(i)]
    \item When $\{ \sigma_n\}_{n\in \N}$ is a bounding sequence for a subset $\Mc_1 \subset \Mc$ by Remark \ref{rem:trei.sapte} we have $E(\sigma_n) \le E(\sigma_{n+1})$, $n\in\N$, and consequently $\lim\limits_{n} E(\sigma_n)x(\omega) = x,\ \wp$ - a.e. We also have that $\bigcup\limits_{n=1}^\infty E(\sigma_n) H$ is dense in $L^p(\wp, H), p=0,2$.
    \item When $\Mc_1$ is finite, then it has a bounding sequence. Indeed, when $\Mc_1 = \{f_1, \dots , f_m\}$, then $\sigma_n :\ = \{\gamma \in \Gamma :\ |f_j (\gamma)| \le n,\ j = 1,2, \dots, m\},\ n \in \N$ is a bounding sequence, since for $\sigma:\ = \bigcup\limits_{n=1}^\infty \sigma_n$ we have $\Gamma \setminus \sigma \subseteq \bigcup\limits_{j=1}^m \{ \gamma, f_j(\gamma) = \infty \}$ so  $E(\Gamma \setminus \sigma) = O_H$ and $E(\sigma) = E(\Gamma) = J_H$. It is a bounding sequence also for the $*$-subalgebra of $\Mc$ generated by $\Mc_1$.
    \item If $\{\sigma_n\}_{n\in \N}$ is a bounding sequence for $f\in \Mc(\Gamma,\Sigma,E)$ then $f\chi_{\sigma_n} \in B(\Gamma , \Sigma)$ and $I(f\chi_{\sigma_n}),\ n\in \N$ do exist by the preceding part.
  \end{enumerate}
\end{rem}
\begin{thm}
  Let $f \in \Mc$ and
  \begin{equation}\label{eq:domain.I(f)}
    D(\tilde I(f)) = \{ x\in H:\ f \in L^2 (E_{x,x}^\omega), \omega - \text{ a.e. }\}.
  \end{equation}
  The following assertions hold:
  \begin{enumerate}[(i)]
  \item when $x \in H$, then $x \in D(\tilde I(f))$, iff for each bounding sequence  $\{ \sigma_n \}_{n \in \N}$ for $f$, the sequence $I (f\chi_{\sigma_n}) x(\omega)$ converges $\omega$ - a.e. in $H$. 
  \item For $x\in D(\tilde I(f))$ the definition
  \begin{equation*}
  \tilde I (f) x(\omega) :\ = \lim\limits_{n} I (f\chi_{\sigma_n}) x(\omega) \quad \omega - \text{a.e.}
  \end{equation*}
  is correct, i.e. does not depend on the bounding sequence $\{ \sigma_n \}$ for $f$.
  \item $\bigcup_{n=1}^\infty E(\sigma_n) H (\subset D(\tilde I(f)))$ is a core for $\tilde I(f)$ and $E(\sigma_n) I(f) \le I(f) E(\sigma_n) = I(f\chi_{\sigma_n}),\ n \in \N$.
  \end{enumerate}
\end{thm}
The proof can be obtained using the ideas contained in Theorem 4.13 of \cite{Schmud2}.

The extended mapping
\[ \Mc(\Gamma,\Sigma,E) \ni f \mapsto \tilde I(f) \in \Ds^2(\Omega, H) \]
has its properties (which are not difficult to prove) contained in
\begin{prop}
Let $f,g \in \Mc(\Gamma,\Sigma,E),\ \lambda \in {\mathbb C}$. Then
\begin{enumerate}[(i)]
\item $\tilde I(f) + \tilde I(g) \subset \tilde I (f+g)$ ;
\item $\tilde I (\lambda f) = \lambda \tilde I (f)$ ;
\item $ \tilde I (f g) \subset \tilde I(f) \tilde I (g)$ ;
\item $\tilde I(f)^\bullet = \tilde I (\bar f)$.
\end{enumerate}
\end{prop}

\section{Transforming random operators into continuous random operators}

This section deals with two mappings. The first one called \emph{the continuous transform}, which sends the closed densely defined random operators from $G$ to $H$ into continuous random operators from $G$ to $H$, and the second one, called \emph{the random bounded transform}, which sends the closed decomposable random operators into a special class of random bounded operators on $H$.
We start by enlarging the framework from \cite[Sections 5.2 and 7.4]{Schmud2}, in which the closed densely defined operators between the Hilbert spaced $H_1$ and $H_2$, are transformed into continuous linear operators between the same Hilbert spaces $H_1$ and $H_2$. In this respect, the following result will be of use.
\begin{lem}\label{lem:unu}
Let $H_1$ and $H_2$ be two separable complex Hilbert spaces and $T$ a closed densely defined linear operator from $H_1$ into $H_2$. Then
\begin{enumerate}[(i)]
  \item $T^*T$ is a positive selfadjoint operator on $H_1$ and $D(T^*T)$ is a core for $T$;
  \item there exists $(I_{H_1} + T^*T)^{-1} =\ : C_T$ as bijective mapping of $H_1$, which is bounded selfadjoint and satisfies
      \begin{equation}\label{eq:bnd.C_T}
      0 \le C_T \le I_{H_1} ;
      \end{equation}
    \item the operator $Z_T:\ = T C_T^{1/2}$, called the continuous transform of $T$, is a pure contraction\footnote{for a specific study of pure contractions see \cite{Kauf}}
     from $H_1$ into $H_2$ (i.e. it satisfies the strict inequality $\| Z_T x\|_{H_2} < \|x\|_{H_1} ,\ x \in H_1$), which will be also written $Z_T \in \bigl( \Bc(H_1, H_2) \bigr)_1^{pu}$. Moreover, the operator $Z_T$ also satisfies
        \begin{equation} \label{eq:leg.Z_T.cu.C_T}
        I - Z_T^* Z_T = C_T.
        \end{equation}
\end{enumerate}
\end{lem}
\begin{proof}[Sketch of the proof]
  (i) and (ii) are contained in \cite[Prop 3.18 (i)]{Schmud2}, while (iii) results without difficulty as (i) from Lemma 5.8 from \cite{Schmud}.
\end{proof}
Now, as first application of the continuous transform
\begin{equation}\label{eq:cont.transf.C_T}
\Cc\Lc_d (H_1, H_2) \ni T \mapsto Z_T \in \bigl(\Bc(H_1,H_2)\bigr)_1^{pu},
\end{equation}
we obtain a method of transforming closed random operators into Skorohod operators.
Indeed the following theorem holds
\begin{thm}
  Let $A$ be a closed s.o. random operator from $G$ to $H$ (i.e. $A \in \Cc\Rs^2 (\Omega; G, H) = \Cc\Ls_d(G, L^2(\wp, H))$). Then $Z_A = A (I_G - A^*A)^{-1/2}$ is a s.o. Skorohod operator from $G$ to $H$ satisfying
  \[
  \| Z_A x \|_{L^2(\wp, H)} < \| x \|_G,\  x \in G, \ x\ne 0,
  \]
  i.e. $Z_A$ is a pure contraction from $G$ into $L^2(\wp, H)$.
\end{thm}
\begin{proof} Apply the previous Lemma for $H_1 = G$ and $H_2 = L^2 (\wp, H)$. \end{proof}
\begin{rem}\label{rem:supl.propr.transf}
  When applying the above result for $H_1 = H_2 = H$, then the mapping
  \begin{equation}\label{eq:bnd.trans.Z_T}
    \Cc\Lc_d (H) \ni T \mapsto Z_T \in \bigl( \Bc(H) \bigr)_1^{pu}
  \end{equation}
  has supplementary properties, namely:
  \begin{enumerate}[(i)]
    \item it is bijective, its inverse being
    \begin{equation}\label{eq:inv.trans.T_Z}
      \bigl( \Bc(H) \bigr)_1^{pu} \ni Z \mapsto T_Z \in \Cc\Lc_d (H) ,
    \end{equation}
    where
    \begin{equation}\label{eq:inv.trans.expr}
      T_Z :\ = Z(I - Z^*Z)^{1/2} , Z \in \bigl( \Bc (H) \bigr)_1^{pu},
    \end{equation}
    (see also \cite{Kauf});
    \item the operators $T$ and $Z_T$ are also related by
    \[ (I_H + T^*T)^{-1/2} = (I_H - Z_T^*Z_T)^{1/2} ;\ T \in \Cc\Lc_d(H); \]
    \item the mappings \eqref{eq:bnd.trans.Z_T} and \eqref{eq:inv.trans.T_Z} preserve adjoints, selfadjointness and normality (see also (ii) and (iii) from Lemma 5.8 of \cite{Schmud2}).
  \end{enumerate}
\end{rem}
In what follows the mappings \eqref{eq:bnd.trans.Z_T} and \eqref{eq:inv.trans.T_Z} will be used in transforming the closed decomposable random operators into decomposable Skorohod operators and vice versa.
Let $A = A_{\mathbf a}$ be a closed random operator which is decomposable by the measurable field ${\mathbf a} = \{ a_\omega \}_{\omega\in\Omega}$ of operators on $H$.
If for each $\omega \in \Omega$ we denote $b(\omega) = Z_{a(\omega)}$, then obviously ${\mathbf b} = \{ b(\omega) \}_{\omega \in \Omega}$ is a 2-summable field of pure contractions on $H$ and consequently $A_{\mathbf b}$ is a Skorohod operator decomposable by the family ${\mathbf b}$.
Denoting \begin{equation}
A_{\mathbf b} = \Zc A_{\mathbf a},
\end{equation}
we see that the mapping $\Zc$ transforms the closed decomposable random operators $A$ on $H$ into random operators $B$ satisfying
\begin{equation}\label{eq:contr}
\| Bx (\omega) \| < \| x \| ;\quad x \in H, \ \text{ for almost all }\ \omega \in \Omega.
\end{equation}
As in the case of random contractions, it is clear that such random operators are contained in the $*$-algebra $\Ls^2 (\Omega;H)$. So the above defined mapping $\Zc$ transforms the class $\Cc\Dc^0 (\Omega;H)$ of closed decomposable random operators on $H$ into the $*$-algebra $\Ls^2(\Omega;H)$. This is why $\Zc$ will be called \emph{the random bounded transform}. In what follows the random operators satisfying \eqref{eq:contr} will be called \emph{pure contractive random operators}, their subclass being denoted by $\Rs_{pu}(\Omega ; H)$.
\begin{thm}
   The random bounded transform $\Zc$ has the properties
  \begin{enumerate}[(i)]
    \item it is a bijection between the class of closed densely defined random operators on $H$ and the class of decomposable random operators on $H$;
    \item it preserves the random adjoints, i.e.
    \begin{equation}\label{eq:punct.prin.trans.Zc} (\Zc A)^\bullet = \Zc A^\bullet ;\quad A \in \Cc\Dc^0(\Omega ; H) \end{equation}
    \item it preserves the random normality, i.e. $\Zc A$ is continuous random normal if $A$ is random normal;
    \item it preserves the random selfadjointness, i.e. $\Zc A$ is continuous random selfadjoint if $A$ is random selfadjoint.
  \end{enumerate}
\end{thm}
\begin{proof}
(i) As already observed $\Zc$ maps $\Cc\Dc^0 (\Omega; H)$ into $\Rc_{pu} (\Omega ; H)$. Let now $B \in \Rc_{pu} (\Omega ; H)$, i.e. it is a random pure contractive operator on $H$. Then, as observed, $B$ is random decomposable by a $2$-summable field  ${\mathbf b} = \{b(\omega)\}_{\omega \in \Omega}$, i.e.
$B = A_{{\mathbf b}}$, which from \eqref{eq:contr} satisfies
\[ \|b(\omega) x\|= \|A_{{\mathbf b}} x (\omega)\| < \| x\|,\quad x\in H,\ \omega \in \Omega. \]
This means that $b(\omega) \in \bigl( \Bc(H) \bigr)^{pu}_1$ and, consequently, by Remark \ref{rem:supl.propr.transf} (i)
\begin{equation}\label{eq:contr.decomposed}
  a(\omega):\ = T_{b(\omega)} ,\quad \omega \in \Omega,
\end{equation}
provides a measurable field ${\mathbf a} = \{a(\omega)\}_{\omega \in \Omega}$ of operators on $H$ such that $A_{{\mathbf a}}$ is a pre-image of $A_{{\mathbf b}}$, i.e. $B = A_{{\mathbf b}} = \Zc A_{{\mathbf a}}.$ Thus the mapping $\Zc$ from $\Cc\Dc^0 (\Omega; H)$ to $\Rc_{pu} (\Omega ; H)$ is onto (surjective). For the injectivity, let $A_{{\mathbf a}'}$ be another pre-image in $\Cc\Dc^0 (\Omega ; H)$ of $B = A_{{\mathbf b}}$, i.e. $A_{{\mathbf b}} = \Zc A_{{\mathbf a}'}$. When ${\mathbf a}' = \{ a'(\omega)\}_{\omega \in \Omega}$, then by definition we have that $b(\omega) = Z a'(\omega),\ \omega \in \Omega$, wherefrom, also by Remark  \ref{rem:supl.propr.transf}, we have $a'(\omega) = L_{b(\omega)},\ \omega \in \Omega$, which compared with \eqref{eq:contr.decomposed} leads to $a(\omega) = a'(\omega),\ \omega\in\Omega$. Thus, the mapping $\Zc$ is a bijection of $\Cc\Dc^0(\Omega ; H)$ onto $\Rc_{pu}(\Omega ; H)$, having as inverse the mapping $\Tc$ from $\Rc_{pu}(\Omega ; H)$ onto $\Cc\Dc^0(\Omega ; H)$, defined by $\Tc A_{{\mathbf b}} = A_{{\mathbf a}}$, where $a(\omega) = T_{b(\omega)}$.

(ii) Let $A\in \Cc\Dc^0 (\Omega ; H)$ and $\ba = \{ a(\omega) \}_{\omega \in \Omega} $ be the measurable field such that $ A = A_\ba $. Denote $\ba^* = \{ a^*(\omega)\}_{\omega\in\Omega}$ and ${\mathbf b} = \{ b(\omega) \}_{\omega \in \Omega}$, where $ b(\omega) = Z_{a(\omega)},\ \omega \in \Omega$. Then, by Remark \ref{rem:supl.propr.transf} we have $b(\omega)^* = Z_{a(\omega)^*},\ \omega\in\Omega$. Then, since by definition \eqref{eq:bnd.trans.Z_T} we have $A_{{\mathbf b}} = \Zc A_\ba$ and $A_{{\mathbf b}^*} = \Zc A_{\ba^*}$. But by Remark \ref{rem:Skorohod.adjoint} and Theorem \ref{thm:propr.rnd.op} we have $A_{\ba^*} = (A_\ba)^\bullet$ and $A_{{\mathbf b}^*} = (A_{{\mathbf b}})^\bullet = \bigl( \Zc A_\ba \bigr)^\bullet$, which implies $(\Zc A_\ba)^\bullet = \Zc A_\ba ^\bullet$ hence \eqref{eq:punct.prin.trans.Zc} holds.

(iii) Let $A$ be a random normal operator on $H$. Then, by theorem \ref{thm:car.normal} there exists a measurable field $\ba = \{ a(\omega) \}{\omega \in \Omega}$ of densely defined closed normal operators such that $A = A_\ba$. Then by (iii) in Remark \ref{rem:supl.propr.transf} $Z_{a(\omega)} = b(\omega)$ are continuous normal operators, which means that $A_{{\mathbf b}}$, with ${\mathbf b} = \{ b(\omega) \}_{\omega \in \Omega}$ is a continuous random normal operator decomposable by ${\mathbf b}$. Hence $A_{{\mathbf b}} = \Zc A_\ba = \Zc A$ as desired.

(iv) runs analogously.
\end{proof}
\begin{rem}\label{rem:patru.doi}
  Observing that all the values $B = A_{{\mathbf b}} ;\ {\mathbf b} = \{ b(\omega)\}_{\omega \in \Omega}$ of $\Zc$ are pure contractive random operators on $H$, i.e. they satisfy condition \eqref{eq:contr}, it is obvious that each operator $b(\omega) ,\ \omega \in \Omega$ satisfies the condition
  \begin{equation}\label{eq:contr.spectr.cond}
    \sigma(b(\omega)) \subset \bar \D, \quad \omega \in \Omega ,
  \end{equation}
  where $\D$ is the unit disc in the complex plane $\C$.
\end{rem}

\section{The spectral theorems for general normal or selfadjoint random operators}

Since the result of Remark \ref{rem:patru.doi} also holds in the particular cases when $b(\omega), \omega \in \Omega$ are normal or selfadjoint, the random spectral theorem for continuous random normal or continuous random selfadjoint operators has a much simpler form than in theorem 2.4 from \cite{Thang4}.  To be more explicit let us recall this result, but with the corresponding notation which we use in the present paper. At the same time, since by Example in Section 4 of \cite{Thang4}, each random spectral measure generates a generalized spectral measure and by Theorem 4.4 in \cite{Thang4} each generalized random spectral measure has a modification which is a random spectral measure, we state the mentioned spectral theorem directly with respect to a generalized random spectral measure. Moreover, since the spectrum of all ``components'' of a pure contractive random operator is contained in the unit disc $\bar \D$ or in case of selfadjointness, obviously in $[-1, 1]$, the limit in this cases avoided. Consequently, Theorem 2.4 from \cite{Thang4} takes the following form.
\begin{thm}\label{thm:contr.rnd.ms}
  \begin{enumerate}[(i)]
    \item Let $B$ be a pure contractive normal random operator. There is a uniquely determined random projection operator valued measure $F$ on $(\bar\D, \Bc or \bar \D, H)$ such that for each $x \in H$ and $\omega \in \Omega$ we have
        \begin{equation}\label{eq:int.repr.norm}
          (Bx)(\omega)  = \int_{\bar \D} \lambda d F_x^\omega (\lambda) ,
        \end{equation}
        where $F_x^\omega (\sigma) = F(\sigma) x(\omega) ,\ \sigma \in \Bc or (\bar \D)$
    \item Let $B$ be a pure contractive selfadjoint random operator. Then there is a random projection operator valued measure $F$ on $([-1, 1], \Bc or [-1, 1], H)$ such that for each $x \in H$ and $\omega \in \Omega$ it holds
        \begin{equation}\label{eq:int.repr.selfadj}
          (Bx)(\omega) = \int_{-1}^1 t d F_x^\omega (t) ,
        \end{equation}
        where $F_x^\omega (\tau) = F(\tau) x (\omega) ,\ \tau \in \Bc or [-1, 1]$.
  \end{enumerate}
\end{thm}
Now, before stating the final result, let's briefly discuss the transformation of 
random spectral measures.

Let $(\Gamma, \Sigma, E)$ be a random projection operator valued measure and $\varphi$ a mapping of $\Gamma$ onto the set $\Gamma'$, while $\Sigma'$ consists of the inverse images of elements of $\Sigma$, i.e.
\[ \Sigma' = \{ \tau \subset \Gamma': \varphi^{-1} (\tau) \in \Sigma \}. \]
Observing that the mapping
\begin{equation}\label{eq:tranf.rnd.ms}
  F(\tau) = E (\varphi^{-1} (\tau )) ,\quad \tau \in \Sigma' ,
\end{equation}
is a random projection operator valued measure on $(\Gamma',\Sigma', H)$, we can state
\begin{thm}
  For each $g \in \Mc'= \Mc (\Gamma', \Sigma', F)$, we have $g \circ \varphi \in \Mc = \Mc(\Gamma, \Sigma, E)$ and
  \begin{equation}\label{eq:int.form.transf}
    \int_{\Gamma'} g(\gamma') d F_x^\omega (\gamma') = \int_\Gamma g(\varphi(\gamma)) d E_x^\omega (\gamma) ,\quad x\in H , \ \omega \in \Omega.
  \end{equation}
\end{thm}
\begin{proof}
First \eqref{eq:tranf.rnd.ms} implies easily that $g\circ \varphi \in \Mc$. Having in view the transformation formula for scalar measures we have
\[ \int\limits_{\Gamma'} |g(\gamma')|^2 d\langle F(\gamma') x (\omega), x \rangle = \int_{\Gamma} | g(\varphi (\gamma))|^2 d \langle E(\gamma) x(\omega) , x \rangle , x \in H \]
and
\begin{multline*}
  \int\limits_{\Gamma'} g(\gamma') d\langle F(\gamma') y (\omega), y \rangle = \int_{\Gamma}  g(\varphi (\gamma)) d \langle E(\gamma) y(\omega) , y \rangle , \\
 y \in H \text{ s.t. } g \text{ is } F_{y,y}^\omega - \text{integrable},
\end{multline*}
wherefrom we infer that $D(I_F(g)) = D(I_E (g\circ \varphi))$, which implies $\langle I_F (g) y(\omega), y \rangle = \langle I_E (g \circ \varphi) y(\omega) , y \rangle,\ y \in D(I_F(g)) = D(I_E (g\circ\varphi))$ and, by a polarization type formula, we have $I_F (g) = I_E (g\circ\varphi)$.
\end{proof}
\begin{thm}
  \begin{enumerate}[(i)]
  \item Let $A$ be a closed random normal operator on $H$. Then there is a uniquely determined random projection operator valued measure $E$ on $(\C, \Bc or \C, H)$ such that
      \begin{equation}\label{eq:minus.unu}
      (Ax)(\omega) = \int_\C z d \langle E(z) x(\omega), x \rangle ,\quad x \in D(A), \omega \in \Omega.
      \end{equation}
  \item Let $A$ be a closed random selfadjoint operator on $H$. Then there exists a uniquely determined random projection operator valued measure $E$ on $(\R , \Bc or \R, H)$ such that
      \begin{equation}\label{eq:ultima}
        (Ax)(\omega) = \int_\R t d \langle E(t) x (\omega), x \rangle ,\quad x\in D(A),\  \omega \in \Omega.
      \end{equation}
 \end{enumerate}
\end{thm}
\begin{proof}
(i) Let $\Zc (A) = B$ be the random bounded transform of $A$. Since, in this case $B$ is a pure contractive normal random operator, by applying Theorem \ref{thm:contr.rnd.ms} (i), there exists a uniquely determined r.p.o.v. measure $F$ on $(\bar \D, \Bc or(\bar \D), H)$ such that, for each $x \in H$ and $\omega \in \Omega$ we have
\[ (\Zc (A) x)(\omega) = \int\limits_{\bar \D} \lambda d F_x^\omega (\lambda), \]
with $F_x^\omega(\sigma) = F(\sigma) x (\omega),\ \sigma\in\Bc or (\bar \D)$, or briefly
\[ \Zc (A) = \int\limits_{\bar \D} \lambda d F (\lambda) .\]
Now, to the r.p.o.v. measure $(\bar \D , \Bc or (\bar \D), F)$ we can associate, as before, the spectral integral  with respect to $F$ out of bounded measurable complex functions, i.e. from functions of the class $B(\bar \D , \Bc or (\bar \D))$ with respect to the r.p.o.v. measure $F$, which leads to the continuous functional calculus associated to the random operator $\Zc (A)$:
\[ I_F (f) = f (\Zc (A)) = \int\limits_{\bar \D} f(\lambda) d F (\lambda),\quad f \in B (\bar \D , \Bc or (\bar \D) ).\]
Secondly, we can continue with the extended random integral with respect to the r.p.o.v. measure $(\bar \D ,  \Bc or (\bar \D), F)$ out of the (not necessarily bounded) $F$ - a.e. finite measurable functions of the class $\Mc (\bar \D . \Bc or (\bar \D) , F)$:
\[  I_F (g) = g(\Zc (A)) = \int\limits_{\bar \D} g (\lambda) d F (\lambda) ,\quad g \in \Mc (\bar \D , \Bc or (\bar \D) , F), \]
where each $I_F(g)$ is a random (not necessarily continuous) normal operator on $H$.

Now, to obtain $A$ from $g(\Zc(A))$ we put $g_1 (\lambda) = \dfrac{\lambda}{(1-\lambda^2)^{1/2}}$, since it is not difficult to prove that $F$ is supported in $\D$ (due to the restrictions which are satisfied by the operators from the image $\Rc_{pu}$ of $\Zc$), applying the transformation of the r.p.o.v. measure $F$ on $(\bar \D, \Bc or (\bar \D))$ into a r.p.o.v. measure $E$ on $(\C , \Bc or \C)$ with respect to the map $\D \ni \lambda \mapsto g_1 (\lambda) \in \C$, we finally have
\[ A = I_F (g_1) = \int\limits_{\D} g_1 (\lambda) d F (\lambda) = \int\limits_\C z d E(z), \]
where $E(\sigma) = F (g_1^{-1} (\sigma) ),\ \sigma \in \Bc or ( \C)$, which leads to \eqref{eq:minus.unu}. Now, the uniqueness of $E$ will be obtained by a straightforward reasoning.

(ii) runs analogously.
\end{proof}

\subsection*{Final remark}
In an analogue way we can transpose the Cayley transform to the framework of closed random decomposable operators, to obtain a bijective correspondence between the unitary random operators and selfadjoint random operators. Then we can apply the random spectral theorem for unitary random operators to obtain a random spectral theorem for selfadjoint random operators.
We can then apply such results to the study of random unitary groups and of random generators and random cogenerators as well as their connection. Such a study will be conducted in a forthcoming paper.

\end{document}